\documentclass[10pt]{article}
\usepackage{graphicx}
\usepackage{latexsym}
\usepackage{amsmath}
\usepackage{amssymb}
\usepackage{amsthm}
\usepackage{multirow}
\usepackage{authblk}
\usepackage[spaces,hyphens]{url}
\usepackage{bm}
\usepackage[margin=1in]{geometry}
\usepackage{hyperref}
\usepackage{tikz}
\usepackage{multicol}
\usepackage{mathtools}

\def \N {{\mathbb N}}
\def \Z {{\mathbb Z}}

\def \H {{\mathbb H}}

\def \E {{\mathcal E}}
\def \G {{\mathcal G}}

\def \V {{\mathcal V}}

\def \al {{\alpha}}

\def \Ga {{\Gamma}}

\usepackage[colorinlistoftodos,prependcaption,textsize=tiny]{todonotes}
\usepackage{xargs} 
\newcommandx{\unsure}[2][1=]{\todo[linecolor=red,backgroundcolor=red!25,bordercolor=red,#1]{#2}}
\newcommandx{\change}[2][1=]{\todo[linecolor=blue,backgroundcolor=blue!25,bordercolor=blue,#1]{#2}}
\newcommandx{\newidea}[2][1=]{\todo[linecolor=green,backgroundcolor=green!25,bordercolor=green,#1]{#2}}
\newcommandx{\improvement}[2][1=]{\todo[linecolor=pink,backgroundcolor=pink!25,bordercolor=pink,#1]{#2}}

\newtheorem{theorem}{Theorem}[subsection]
\newtheorem{cor}[theorem]{Corollary}
\newtheorem{lemma}[theorem]{Lemma}
\newtheorem{pro}[theorem]{Proposition}

\newtheorem{definition}[theorem]{Definition}

\newtheorem{ex}[theorem]{Example}

\usepackage{setspace}
\usepackage[square,sort,comma,numbers]{natbib}
\usepackage[english]{babel}
\usepackage[euler]{textgreek}
\usepackage{stmaryrd}
\usepackage{mathtools}
\usepackage{pgf,tikz}
\usepackage{tikz}
\usepackage{tkz-graph}
\usepackage{array}
\tikzstyle{vertex}=[circle,draw, inner sep=0pt, minimum size=4pt] 
\newcommand{\vertex}{\node[vertex]}
\usetikzlibrary{decorations.markings}
\title{Spectral Bounds of the Generating Graph of $\Z_n$}
\author[1]{Kavita Samant}
\author[2]{A. Satyanarayana Reddy}
\affil[1,2]{Department of Mathematics\\Shiv Nadar Institution of Eminence, Delhi-NCR, India.}
\affil[1]{Corresponding author. E-mail id(s): ks299@snu.edu.in;}
\affil[2]{Contributing author: satya.a@snu.edu.in}

\date{}

\begin{document}
\maketitle
\begin{abstract}
    Let \( G \) be a group. A group is said to be \( k \)-\emph{generated} if it can be generated by its \( k \) elements. A generating set of \( G \) is called a \emph {minimal generating set} if no proper subset of it generates \( G .\) A minimal generating set of a group can have different sizes. The \emph {generating graph} \(\Gamma (G) \) of a group \( G \) is defined as a graph with the vertex set \( G \), where two distinct vertices are adjacent if they together generate \( G .\) This graph is particularly useful when studying 2-generated groups.

    In this context, consider the group \( G = \mathbb{Z}_n \), the integers modulo \( n .\) In this paper, we explore various graph-theoretic properties of the generating graph \( \Gamma(\Z_n) \) and investigate the spectra of its adjacency and Laplacian matrices. Additionally, we explicitly determine the set of all possible  minimal generating sets of \( \mathbb{Z}_n \) of size $k.$
\end{abstract}
\textbf{Keywords.} Generating graphs; Comaximal graphs; Cyclic groups; Adjacency matrix; Laplacian matrix; Spectrum.\\
\textbf{Mathematics Subject Classification.} 05C25, 05C50, 20D60.

\section{Introduction}
A group \( G \) is called \emph {cyclic} if it can be generated by a single element. A finite cyclic group of order \( n \), is isomorphic to the group of integers modulo \( n \), denoted by \( \mathbb{Z}_n \), where we simply write \( \mathbb{Z}_n = \{0, 1, 2, \dots, n-1\} .\) The set of all generators of \( \mathbb{Z}_n \) is precisely the set \( \{a \in \mathbb{Z}_n \,|\, \gcd(a, n) = 1\} \), which is denoted by \( U(n) .\) The set \( U(n) \) forms a group under the multiplication modulo \( n .\)

A group \( G \) is said to be \( k \)-\emph {generated} if there exists a set of \( k \) elements that generates \( G .\) Clearly, if a group is \( k \)-generated, it is also \( l \)-generated for all \( k \leq l \leq |G| .\) For \(G= \mathbb{Z}_n \), the group is \( k \)-generated for any \( 1 \leq k \leq n .\) However, these generating sets are not necessarily minimal. 

In this article, we determine minimal generating sets of all possible sizes of \( \mathbb{Z}_n .\) However, our primary focus will be on identifying and analyzing 2-generated sets of \( \mathbb{Z}_n .\)

The {\em generating graph} of $G,$ denoted by $\Ga(G),$ is a graph whose vertex set is $G,$ and any two distinct vertices are adjacent if they generate $G.$ Thus, it is natural to consider only two generated groups; otherwise, the generating graphs will be empty. The fundamental idea behind generating graphs is the generation of groups by its two elements, which was first studied in a probabilistic way. The story starts in 1892, with a conjecture of {E. Netto}~\cite{netto}, that most pairs of elements of the symmetric group generate the symmetric group or the alternating group with a probability near 1 for all sufficiently large $n.$
This conjecture was proved by Dixon~\cite{dixon} in 1969. In his paper, Dixon makes rather a general conjecture that two elements chosen at random from a finite simple group $G$ generates $G$ with probability $\rightarrow 1$ as $|G|\rightarrow \infty.$
The idea can be stated as any two elements selected at random from a finite group $G$ of cardinality $n,$ generate $G$ with the probability $P(G),$ which is defined as
$$P(G)=\frac{|\{(x,y)\in G\times G:\langle x,y\rangle=G\}|}{\binom{n}{2}}.$$  

Later, concerning the study of the structure of two generated groups and their generating pairs, the concept of generating graphs was introduced by Andrea Lucchini and Attila Mar\'oti in~\cite{some,clique}. In these papers, they investigated various graph theoretic properties of generating graphs of finite groups and proposed many questions. Former for finite simple groups, many deep results have been studied using the probabilistic notion by Liebeck and Shalev in~\cite{Lieback}, Guralnick and Kantor in~\cite{Guralnick} and Andrea Lucchini and E. Detomi in~\cite{MR1987022} which can be equivalently stated as theorems that ensure that $\Ga(G)$ is a rich graph. 

In the literature, for finite dihedral and dicyclic groups, the spectrum of the various associated matrices of the generating graph has been completely determined in~\cite{self,self2}. In this paper, we are interested in studying the spectral properties for the finite cyclic groups.

Recall that for any finite simple undirected graph $\G$ with the vertex set $V(\G)=\{v_1,\dots,v_m\}.$ The adjacency matrix $A(\G)$ is an $m\times m$ matrix, where each entry is either 1 or 0 depending on whether the vertices that correspond to a row and a column are adjacent or not, respectively. The Laplacian matrix $L(\G):=D(\G)-A(\G),$ where $D(\G)=\text{diag}(\alpha_1,\alpha_2,\dots,\alpha_m)$ is the degree matrix. Both the matrices are real and symmetric. Thus, all their eigenvalues can be arranged in a non-increasing order.

 Throughout this article, we represent $\Z_n$ as the additive groups of integers modulo $n.$ In this paper, we study the generating graph $\Ga(\Z_n)$ with the vertex set $\{0,1,2,\dots,n-1\}.$ We use \(\E_n\) as an abbreviation for \(\Ga(\Z_n).\) Our objective is to explore the graph-theoretical properties of \(\E_n\) to gain a deeper understanding of its structure. 
We are primarily interested in determining the spectra of the associated matrices with the graph, specifically adjacency and Laplacian matrices.

The group \(\Z_n\) possesses a ring structure under addition and multiplication modulo \(n.\) It has been observed that for any \(n \in \N\), the graph \(\E_n\) is isomorphic to the co-maximal graph of \(\Z_n.\) The co-maximal graph, denoted by \(\Gamma_C(R)\), of a commutative ring \(R\) with identity is defined as a graph whose vertices correspond to the elements of \(R.\) Two distinct vertices \(a\) and \(b\) are adjacent if and only if \(Ra + Rb = R.\) 

The co-maximal graph of \(\mathbb{Z}_n\) has been investigated in~\cite{Add,Lap1,Lap2}, where the spectra of its adjacency and Laplacian matrices were analyzed. However, a complete characterization of the spectrum remains unresolved. In this paper, we employ matrix analysis techniques to establish tighter bounds for the eigenvalues of the adjacency and Laplacian matrices.

For any \(n \in \N\), we denote \(n_0\) as the square-free part of \(n.\) We demonstrate that understanding the structure of \(\E_{n_0}\) is sufficient for analyzing \(\E_n.\) In other words, we see that the graphs \(\E_n\) and \(\E_m\) with the same set of prime divisors (\(n_0 = m_0\)) exhibit similar behavior. Consequently, we show that determining eigenvalue bounds for \(\E_{n_0}\) will be sufficient to figure out bounds in general, for any $m,$ with $m_0=n_0.$
In the article, concerning the Laplacian matrix we show that the eigenvalues of \(\E_n\) are scalar multiple of the eigenvalues of \(\E_{n_0}\), with the scalar factor \(\frac{n}{n_0}.\)

The paper is organized into six sections. In Section~\ref{sec:structure}, we give explicit sets of the minimal generating sets of \(\Z_n\) of various sizes. Section~\ref{sec:graph} focuses on the structure of the generating (or co-maximal) graph of \(\Z_n.\) Sections~\ref{sec:adjacency} and~\ref{sec:laplacian} address the spectra of the adjacency and Laplacian matrices, respectively. Finally, in the last section, we conclude with a summary of results.

\section{Minimal Generating Sets of $\Z_n$}\label{sec:structure}
A generating set of a group is called \emph {minimal} if removing any element from the set results in a subset that no longer generates the group. The sizes of minimal generating sets can vary within a group. For instance, the group \( \mathbb{Z}_6 \) has minimal generating sets such as \( \{1\} \) and \( \{2, 3\} .\) In contrast, for \( \mathbb{Z}_p \), where \( p \) is a prime, a minimal generating set can be of size one only.

This section focuses on identifying all possible sizes of minimal generating sets for the group \( \mathbb{Z}_n .\) In 1935, Philip Hall introduced a formula to determine the number of ways to generate the symmetry group of an icosahedron using a specific number of its elements. In the paper~\cite{Hall}, he demonstrated that this method could be extended to any finite group, provided its subgroup structure is sufficiently understood. For a finite group \(G\), Hall defined a function that counts the number of ways to generate \(G\) using \(k\) elements. This function, now known as the generalized Eulerian function, is denoted by \(\Phi_k(G).\) 

In the thesis~\cite{Thesis}, the author derives explicit formulas for \(\Phi_k(G)\) for specific classes of groups, including cyclic groups, dihedral groups, and \(p\)-groups. However, our focus will be on minimal generating sets of \(\mathbb{Z}_n\), which are independent sets that generate the group. We also develop an explicit number-theoretic method to identify all such minimal generating sets and possible sizes. Once the complete set of minimal generating sets is determined, its cardinality can be easily calculated. In addition that cardinality is always less than equal to $\Phi_k(G).$
We begin by introducing the essential notations and mathematical symbols necessary for understanding our work. To maintain coherence and clarity, we have adopted the following notations.

Let $n=p_{1}^{e_{1}} p_{2}^{e_{2}}\dots p_{r}^{e_{r}}$ be the prime factorization of $n, $ where $p_1<p_2< \dots < p_r$ are distinct primes and ordered. Let $n_0$ be the square free part of $n,$ {\it i.e.,} $n_0=p_1p_2\cdots p_r.$ We denote $\tau(n), \omega(n)$ and $\varphi(n)$ as the number of positive divisors of $n,$ the number of distinct prime divisors of $n$ and the Euler totient function of $n,$ respectively.
For each $n,$ we can find all possible minimal generating sets of size m, where $m\in \N.$

We begin by identifying the possible generating pairs in the group \( \mathbb{Z}_n \),{\it i.e.,} pairs \( a, b \in \mathbb{Z}_n \) such that \( \langle a, b \rangle = \mathbb{Z}_n .\) By definition, \( \langle a, b \rangle = \{ax + by \mid x, y \in \mathbb{Z}\} .\) Hence, \( \langle a, b \rangle = \mathbb{Z}_n \) if and only if \( \gcd(a, b, n) = 1 .\) Note that $a\in \Z_n$ with $\gcd(a,n)=1,$ will always be a generating pair. But if $a$ or $b$ does not belong to $U(n),$ then we have to figure out which ones are those.

Let $\tilde G_2$ be the set of ordered pairs of  $\Z_n$ that generates the entire group. Hence we can conclude that
\begin{align}\label{eq:pair}
 \tilde G_2&=\{(a,b) \,|\, a\neq b \in \Z_n \text { such that } \gcd(a,b)\in U(n)\}.
\end{align}
To understand the generating pairs well, we partition the set \(\{0, 1, 2, \dots, n-1\}\) in such a way that the elements within each subset exhibit similar behaviour with respect to the specified generating property.
We define a relation $\sim$ as follows.
\begin{equation}\label{eq:rel}
 a\sim b \text{ whenever } \gcd(a,n_0)=\gcd(b,n_0).
\end{equation}
It is an equivalence relation. The number of equivalence classes equals the number of divisors of $n_0,$ denoted by $\tau(n_0),$ and exactly equal to $2^r.$ Now corresponding to each divisor $d$ of $n_0,$ the equivalence class is given by 
\begin{equation}\label{eq:class}
 [d]=\{a\in \Z_n|\gcd(a,n_0)=d\}. 
\end{equation}
For a general $n,$ the best representatives of the equivalence classes are the divisors of $n_0.$ One can observe a one-to-one correspondence between the set of divisors of $n_0$ and the set of all binary sequences of length $r,$ where $r$ is the total number of prime factors of $n_0.$ For example, if $n=24,$ then $n_0=6.$ If $d_{(ab)}|6,$ then $d_{(ab)}=2^a3^b,$ where $a,b\in \{0,1\},$ that is, the corresponding binary sequence with respect to $d_{(ab)}$ is $(ab).$ In particular, $3|6$ and we denote $3=d_s,$ where $s=(01).$ 
The following result concludes the cardinality of each equivalence class.
 \begin{theorem}
 For a given $n,$ the size of an equivalence class of the set $\{0,1,2,\dots,n-1\}$ is given by 
    $$|[d_s]|=
 \frac{n}{n_0}\varphi\left(\frac{n_0}{d_{s}}\right).$$
 \end{theorem}
 \begin{proof}
 Let us analyze each class $[d_{s}],$ where $d_s$ corresponds to a divisor of $n_0,$ and $s$ represents a binary sequence. Note that we can also see Equation~\ref{eq:class} as 
    \begin{equation}\label{eq:class2}
 [d_{s}]=\left\{ad_s\,|\, a\in \Z_n \text{ with }\gcd\left(a,\frac {n_{0}}{d_{s}}\right)=1 \right\}.
 \end{equation} 
 In the case when $s=(00\ldots0),$ the corresponding divisor is $1,$ and thus the class $$[1]=\{a\in \Z_n \,|\,\gcd(a,n_{0})=1 \}=\{a\in \Z_n \,|\,\gcd(a,n)=1 \}=U(n).$$ Therefore there are exactly $\varphi(n)$ elements in the class $[d_{s}].$ Note that $\varphi(n)=\frac{n}{n_0}\varphi(n_0).$ However, if $s\neq (00\ldots0),$ then by Equation~\ref{eq:class2}, we get
    \begin{equation}\label{eq:csize}
 |[d_{s}]|= \frac{n}{n_0}\varphi\left(\frac{n_0}{d_{s}}\right).
        \end{equation}
    \end{proof}
Now, we have the following implication. This result will later help to relate the cases when $n$ and $m$ have the same set of prime divisors. 

\begin{cor}
 For $n=n_0,$ the size of an equivalence class of the set $\{0,1,2,\dots,n_0-1\}$ is given by 
    $$|[d_s]|=\begin{dcases}
      \varphi(n_0)& d_s=1, \\
      \varphi\left(\frac{n_0}{d_{s}}\right)&\text{otherwise}.
      \end{dcases}$$
   \end{cor} 
   \begin{lemma}\label{lem:ds}
 If $s\neq (00\dots 0),$ then no two elements in the class $[d_s]$ will be the generating pairs.
   \end{lemma}
   \begin{proof}
 If $s\neq (00\dots 0),$ any two element of $[d_s]$ have a common multiple $d_s$ which divides $n.$ Thus, it cannot be a generating pair.
   \end{proof}
 Let us find $\tilde G_2$ in the following examples.
   \begin{ex}
 Let $n=32=2^5$ and $n_{0}=2.$ Then we have exactly two equivalence classes with respect to the relation $\sim,$ which are as follows.
   \begin{align*}
 [1]&=\{1,3,5,7,9,11,13,15,17,19,21,23,25,27,29,31\}\\
 [2]&=\{0,2,4,6,8,10,12,14,16,18,20,22,24,26,28,30\}.
   \end{align*}
 Note that $\gcd(1,2)=1,$ then by the above lemma, we have
    $$\tilde G_2=[1]\times[1]\bigcup [1]\times[2]\bigcup [2]\times[1].$$
   \end{ex}
 Let us take the case when $n$ has more than one prime divisor. 
   \begin{ex}
 Let $n=24=2^{3}.3.$ The equivalence classes are as follows.
   $$[1]=\{1,5,7,11,13,17,19,23\},\quad
 [2]=\{2,4,8,10,14,16,20,22\}$$
   $$ [3]=\{3,9,15,21\},\quad
 [6]=\{0,6,12,18\}.$$
 In this example, the set $\tilde G_2$ can be seen as $$\tilde G_2=\left(\bigcup\limits_{d_s\mid n_0}[d_s]\times[1]\right)\bigcup\left(\bigcup\limits_{d_s\mid n_0}[1]\times[d_s]\right)\bigcup [2]\times[3]\bigcup[3]\times[2].$$
   \end{ex}
Having understood the nature of ordered pairs in the set \( \tilde{G}_2 .\) However, we are interested in finding the minimal generating sets of size 2. The generating behaviour of a set is independent of in which order the elements are. Therefore, we denote $G_2,$ as the set of minimal generating sets of size 2, which is given by
\begin{align*}
 G_2&=\{\{g_1,g_2\} \,|\, g_1, g_2 \in \Z_n\setminus U(n) \text { such that } \gcd(g_1,g_2)\in U(n)\}.\\
    &=\{\{g_1,g_2\}\,|\,g_i\in [d_{s_i}] \text{ where } d_{s_i}\neq 1 \text{ and } \gcd(d_{s_1},d_{s_2})= 1\}.
\end{align*}
In general, we define $G_k$ as the set of all minimal generating sets of size $k.$ 

\begin{theorem}
 For a given \( n\), the size of a minimal generating set, say \( k \), satisfies \( k \leq r \), where \( r \) denotes the number of distinct prime factors of \( n .\) Moreover, for $r>2,$ the set can be explicitly described as follows.
 { \footnotesize $$G_k=\Bigg\{\{g_1,\dots,g_k\}\,|\, g_i\in [d_{s_i}]\text{ with } \gcd(d_{s_1},\dots, d_{s_k})=1 \text{ but } \gcd(d_{i_1},\dots,d_{i_{k-1}})\neq 1 \text{, where } i_j\in \{{s_1},\dots, {s_k}\}\Bigg\}.$$}
\end{theorem}
\begin{proof}
Let $n\in \N$ with $r$ distinct prime divisors. We first prove that $\Z_n$ has minimal generating sets of size at most $r.$ We prove it by contradiction.
Assume that $\Z_n$ has a minimal generating set $S$ of size $r+1.$ Let's say $S=\{g_1,g_2,\dots, g_{r+1}\}.$ Note that the minimality of $S$ implies that there is no subset of size $r,$ which is a minimal generating set of $\Z_n.$ Thus, every subset of $S$ with size $S$ will share a common prime divisor among its elements. Since there are exactly $(r+1)$ subsets of $S$ of size $r.$ Thus, we must have $(r+1)$ distinct prime divisors of $n,$ otherwise $\gcd(g_1,\,g_2,\,\dots ,g_{r+1})=p$ for some common prime $p,$ which contradicts the minimality of $S.$ In addition, note that we have only $r$ distinct primes. Thus, it cannot be possible. Hence $k\leq r.$

We already determined the minimal generating set of size 2; now, to give an expression for a minimal generating set of size greater than 2, it is necessary that $r>2.$ Note that if $\langle\{g_1,g_2,\dots ,g_k\}\rangle=\Z_n,$ then there exists $x_1,x_2,\dots, x_n\in \N$ such that $$1=g_1x_1+g_2x_2+\dots +g_nx_n.$$ Furthermore, no subset of $\{g_1,g_2,\dots ,g_k\}$ can be a minimal generating set with a size smaller than $k.$ Thus, we can define $G_k$ as
 {\footnotesize$$G_k=\Bigg\{\{g_1,\dots, g_k\}\,|\, g_i\in Z_n\setminus U(n) \text{ with }\gcd(g_1,\dots, g_k)\in U(n)\text{ but } \gcd(g_{i_1},\dots,a_{i_{k-1}})\notin U(n) \text{, where }i_j\in \{1,\dots, k\}\Bigg\}.$$
 }
 Note that each $g_i$ belongs to an equivalence class; thus, the set $G_k$ contained in the union of the cartesian product of  $k$ distinct equivalence classes. Also, $g_i$ does not belong to the class of the type $[p_i]$ and $[n_0],$ where $p_i's$ are prime divisors of $n,$ so we exclude those classes. We define $$T(n)=\Bigg\{\bigcup_{i=1}^{r}[p_i]\cup [n_0]\cup [1]\Bigg\}.$$ Now we can modify the set $G_k$ as follows.
 {\footnotesize
\begin{align*}
 G_k=\Bigg\{\{g_1,\dots, g_k\}\,|\, g_i\in [d_{s_i}]\notin T(n) \text{ with }\gcd(d_{s_1},\dots, d_{s_k})=1 \text{ but }\gcd(d_{i_1},\dots,d_{i_{k-1}})\neq 1 \text{ ,where } i_j\in \{s_1,\dots, s_k\}\Bigg\}.
\end{align*}}
\end{proof}
The size of each $G_k$ can be easily computed once we know which $d_i's$ plays a role in the group generation. Let us illustrate the above discussion in the following example. 
\begin{ex}
Let $n=p_1^{e_1}p_2^{e_2}p_3^{e_3},$ where $p_i's $ are distinct primes and $e_i's\in \N.$ Then
$G_1=U(n),$ and therefore $$|G_1|=\varphi(n).$$ The set $G_2$ is given by
$$G_2=\bigcup\limits_{i\neq j}[p_i]\times [p_j],$$ and using Equation~\ref{eq:csize}, we get $$|G_2|=\left(\frac{n}{n_0}\right)^2\varphi(n_0)\left(\varphi(p_1)+\varphi(p_2)+\varphi(p_3)\right).$$
Lastly, $G_3=[p_1p_2]\times [p_1p_3]\times [p_2p_3],$ and by using Equation~\ref{eq:csize}, we get $$|G_3|=\left(\frac{n}{n_0}\right)^3\varphi(n_0).$$
\end{ex}

\section{Generating Graph or Comaximal Graph of $\Z_n.$}\label{sec:graph}
A graph \( \G \) is defined as a pair \( (V(\G), E(\G)) \), where \( V(\G) \) represents the set of vertices, and \( E(\G) \) denotes the set of edges connecting pairs of vertices. Two distinct vertices \( u \) and \( v \) are said to be adjacent if an edge exists between them, denoted by \( u \backsim v .\) A graph \( \G \) is termed \emph {simple} if it contains no multiple edges between any pair of vertices and no self-loops. In this paper, we focus exclusively on simple graphs. An \emph {empty} graph is one in which no edges exist between any pair of vertices, whereas a \emph {complete} graph \( K_n \) is a graph, where every pair of distinct vertices is adjacent. 
This section delves into the generating graph of \( \Z_n .\) The vertex set of \( \E_n \), denoted as \( V = \Z_n \), and the edge set, denoted as \( E = \{\{a, b\} \in V \times V \mid \langle a, b \rangle = \mathbb{Z}_n \} .\) It is evident that \( \{a, b\} \in E \) if and only if \( (a, b) \in \tilde{G}_2 .\) The group \( (\mathbb{Z}_n, \oplus_n) \), with multiplication modulo \( n \), possesses a ring structure.

In the existing literature, another notable graph associated with a ring structure is the \emph{co-maximal graph}. Let \( R \) be a commutative ring with identity. The co-maximal graph \( \Ga_C(R) \) is a graph, where the vertices correspond to the elements of \( R .\) Two distinct vertices \( a \) and \( b \) are adjacent if and only if \( Ra + Rb = R, \) where $Ra=\{ra:r\in R\}$ is an ideal generated by an element $a.$ Note that it is equivalent to say that $a$ and $b$ generate the group. Hence, both the graphs are isomorphic for $\Z_n.$ The comaximal graph of $\Z_n$ has been studied in the papers~\cite{Lap1,Lap2,Add}. In this section, we will explore these graphs more concerning the generating graphs. For fundamental definitions and results in graph theory, refer to any standard textbook on graph theory.
\subsection{Graph Decomposition} 
In this subsection, we will decompose the generating graph of a group by a collection of graphs. This decomposition proves to be a valuable tool in our analysis. We first introduce the necessary notations, definitions, and preliminary results to facilitate this approach.

\begin{definition}
 Let $\H$ be a graph with the vertex set $V(\H)=\{1,2,\dots, k\},$ the $\H$-join of a family of graphs $\{\G_i: 1\leq i\leq k\}$ is the graph $\H[\G_1,\G_2,\dots,\G_k]$ with the vertex set $\V=\bigcup\limits_{i=1}^{k}V_i,$ where $V_i$'s are the vertex sets of $\G_i$'s  and the edge set $E=\bigcup\limits_{i=1}^{k} E(\G_i)\bigcup\left\{\bigcup\limits_{\{i,j\}\in E(\H)}\{xy:x\in V(\G_i)\text{ and }y\in V(\G_j)\} \right\}
,$ where $E(\G_i)$ is the edge set of $\G_i.$
\end{definition}

\begin{lemma}\label{lem:class}
 Let $n=p_1^{e_1}p_2^{e_2}\dots p_r^{e_r},$ where $p_i's$ are distinct primes. Let $\{[d_1],[d_2],\dots, [d_{2^r}]\}$ be a partition of $n,$ where $d_i's$ are the divisors of $n_0.$ Then, the graph induced by these classes is either empty or complete. Moreover, the neighbourhood set of each element in a class $[d_i]$ of $\E_n$ is the same.
\end{lemma}
\begin{proof}
 Let $d_1=1.$ We know that $[d_1]$ corresponds to $U(n),$ thus induces a complete graph of degree $\varphi(n)-1.$ However, if $i\neq 1,$ then by Lemma~\ref{lem:ds} the graph induced by $[d_i]$ is empty. From Equation~\ref{eq:rel}, we get that each element in the class will have the same neighbourhood in $\E_n.$
\end{proof}
Now we can see that the generating graph of $\Z_n$ can be decomposed as an $\H$-join of graphs. We take $\H$ as an induced subgraph of the set $\{d_1,d_2,\dots, d_{2^r}\}.$ That is, any two vertices $d_i,d_j\in V(\H)$ are adjacent in $\H$ whenever $\gcd(d_i,d_j)=1.$ To construct the join of $\H,$ for each $1\leq i\leq 2^r,$ we take $\G_i$
as an induced subgraph of $\E_n$ by the equivalence class $[d_i].$ We take $\G_1$ induced by $[d_1=1],$ is $\varphi(n)-1$ regular, however, for each $j\geq 2,$ $ G_j$ is a $0$-regular graph. The neighbourhood set of each element in a class $[d_i]$ of $\E_n$ is same, thus $\H[\G_1,\G_2,\dots,\G_{2^r}],$ has the vertex set $\bigcup\limits_{i=1}^{2^r}V(\G_i)=V.$ However, the adjacency among the vertices of $\H$ decides whether $\G_i$'s are joined or not. Therefore, by the above definition, the edge set of $\H[\G_1,\G_2,\dots,\G_{2^r}]$ is the same as $E.$ So, we can write $$\E_n=\H[\G_1,\G_2,\dots,\G_{2^r}].$$
This is how we can decompose $\E_n$ into simpler known graphs. Now, we can conclude the following. We will carry forward the above notations.
\begin{pro}
    For $\E_n = \H[\G_1, \G_2, \dots, \G_{2^r}]$, the graph $\H$ is connected with a diameter of at most $2$. It is complete or bipartite if and only if $n = p_1^{e_1}$.
\end{pro}
\begin{proof}
 Each vertex of $\H$ is connected to $d_1=1,$ hence connected with diameter at most 2.
 If $n$ has more than one prime divisor, then the set $\{1,p_1,p_2\}$ will form a triangle. Therefore, it cannot be bipartite.
\end{proof}
\subsection{A Quick Analysis of the Graph Structure}

In this section, we will discuss the graph properties. Let us see some examples first.
\vglue 1mm
\begin{ex}
For $n=3,6,8,12,$ the generating graph of $\Z_n$ is given by 
\begin{figure}[!ht]
\centering
\begin{tabular}{cccc}
 \begin{tikzpicture}[scale=0.50]
    \vertex (1) at (2,6) {0};
    \vertex (2) at (3,7) {1};
    \vertex (3) at (4,6) {2};
    \path[-]
    (1) edge (2)
    (2) edge (3)
    (3) edge (1);
\end{tikzpicture}
&\hspace*{10mm}
 \begin{tikzpicture}[scale=0.50]
    \vertex (1) at (0,3) {1};
    \vertex (2) at (4,3) {2};
    \vertex (5) at (0,2) {5};
    \vertex (3) at (4,2) {3};
    \vertex (4) at (2,1.5){4};
   \vertex (6) at (2,5) {0};
    \path[-]
    (1) edge (2)
    (1) edge (4)
    (1) edge (5)
    (2) edge (3)
    (2) edge (5)
    (3) edge (4)
    (5) edge (4)
    (3) edge (5)
    (6) edge (1)
    (6) edge (5)
    (3) edge (1);
\end{tikzpicture}
&\hspace*{10mm}
\begin{tikzpicture}[scale=0.50]
    \vertex (1) at (1,5) {1};
    \vertex (2) at (4,5) {2};
    \vertex (7) at (0,3.5) {7};
    \vertex (3) at (5,3.5) {3};
    \vertex (6) at (-2,2) {6};
    \vertex (4) at (4,2) {4};
    \vertex (5) at (2,1) {5};
    \vertex (9) at (-2,4.5) {0};
    \path[-]
    (9) edge (1)
    (9) edge (3)
    (9) edge (5)
    (9) edge (7)
    (1) edge (2)
    (1) edge (3)
    (1) edge (4)
    (1) edge (5)
    (1) edge (6)
    (1) edge (7)
    (3) edge (2)
    (3) edge (4)
    (3) edge (5)
    (3) edge (6)
    (3) edge (7)
    (5) edge (2)
    (5) edge (4)
    (5) edge (6)
    (5) edge (7)
    (7) edge (2)
    (7) edge (4)
    (7) edge (6);
\end{tikzpicture}
&\hspace*{10mm}
\begin{tikzpicture} [scale=0.70]
    \vertex (5) at (-1,4) {5};
    \vertex (2) at (4,3) {2};
    \vertex (7) at (-1.5,2) {7};
    \vertex (8) at (4,2) {8};
    \vertex (4) at (4,1){4};
    \vertex (6) at (2,5) {6};
    \vertex (1) at (0,4.9){1};
    \vertex (3) at (0.5,-0.2){3};
    \vertex (9) at (2,-0.4) {9};
    \vertex (10) at (3.5,0) {10};
    \vertex (11) at (-1,1) {11};
    \vertex(12) at (3.2,4) {0};
    \path[-]
    (1) edge (10)
    (1) edge (2)
    (1) edge (3)
    (1) edge (4)
    (1) edge (5)
    (1) edge (6)
    (1) edge (7)
    (1) edge (8)
    (1) edge (9)
    (1) edge (12)
    (1) edge (11)
    (5) edge (10)
    (5) edge (2)
    (5) edge (3)
    (5) edge (4)
    (5) edge (6)
    (5) edge (7)
    (5) edge (8)
    (5) edge (9)
    (5) edge (12)
    (5) edge (11)
    (7) edge (10)
    (7) edge (2)
    (7) edge (3)
    (7) edge (4)
    (7) edge (6)
    (7) edge (8)
    (7) edge (11)
    (7) edge (9)
    (7) edge (12)
    (11) edge (10)
    (11) edge (2)
    (11) edge (3)
    (11) edge (4)
    (11) edge (6)
    (11) edge (8)
    (11) edge (9)
    (11) edge (12)
     (2) edge (3)
     (2) edge (9)
     (4) edge (3)
     (4) edge (9)
     (8) edge (9)
     (8) edge (3)
     (10) edge (3)
     (10) edge (9);
    \end{tikzpicture}
\end{tabular}
\caption{The generating graphs $\E_3,\E_6, \E_{8}\text{ and } \E_{12}.$}
\end{figure}
\vglue 2mm
\end{ex}
Some graph properties of $\E_n$ are discussed in various papers~\cite{diameter,generating, planar}, where general results for groups are provided. For completeness and accessibility, we include proofs here, which follow trivially from our approach.

For $a \in U(n) \subseteq V,$ we have $\deg(a) = n-1,$ and the elements of $U(n)$ form a clique of size $\varphi(n).$ Specifically, if $n = p$ (a prime), then $U(n) = V \setminus \{0\},$ and the degree of the identity element is also $n-1,$ making it indistinguishable in the graph. Hence, $\E_p$ is isomorphic to $K_p.$ All vertices in $U(n),$ corresponding to the equivalence class $[d_i] = [1],$ share the degree $n-1,$ which is not coincidental. In fact, all vertices within the same equivalence class $[d_i]$ have the same degree. 

By Lemma~\ref{lem:class}, no two vertices $a, b \in [d_i]$ with $d_s \neq 1$ are adjacent, and their neighborhoods are identical, {\em i.e.}, $N(a) = N(b),$ where $N(x)$ denotes the neighborhood of $x.$ Thus, there are at most $2^r$ distinct vertex degrees. For any $x \in V,$ we find $\deg(x) = \varphi(n)$ if and only if $x$ belongs to the class $[n_0].$ Furthermore, for $x \in [d_i]$ where $d_i \neq 1, n_0,$ the degrees satisfy $\deg(0) < \deg(x) < \deg(1).$ Consequently, the maximum degree $\Delta(\E_n) = n-1,$ and the minimum degree $\delta(\E_n) = \varphi(n).$

For example, if $n = p^k$ with $p$ prime and $k \geq 2,$ then for $x \in V,$ $\deg(x) \in \{(p-1)p^{k-1}, p^k-1\}.$ To illustrate, consider $n = 30.$ 
 \begin{ex}
 Let \(n = 30 = 2 \cdot 3 \cdot 5.\) Then, the vertex set of \(\H\) is \(V(\H) = \{1, 2, 3, 5, 6, 10, 15, 30\}.\) To understand the adjacency of vertices in \(\H\), the table below illustrates the structure of \(\E_{30}.\)
 \begin{figure}[!ht]
    \centering
\begin{tikzpicture}[scale=0.7]
    \vertex (1) at (1,5) {\textcolor{red}{[1]}};
    \vertex (2) at (4,5) {\textcolor{red}{[2]}};
    \vertex (7) at (0,3) {\textcolor{red}{[15]}};
    \vertex (3) at (5,3.5) {\textcolor{red}{[3]}};
    \vertex (6) at (-2,2) {\textcolor{red}{[10]}};
    \vertex (4) at (4,2) {\textcolor{red}{[5]}};
    \vertex (5) at (1,1.5) {\textcolor{red}{[6]}};
    \vertex (8) at (-2,4.5) {\textcolor{red}{[30]}};
    \path[-]
    (8) edge (1)
    (1) edge (2)
    (1) edge (3)
    (1) edge (4)
    (1) edge (5)
    (1) edge (6)
    (1) edge (7)
    (7) edge (2)
    (3) edge (2)
    (4) edge (2)
    (3) edge (4)
    (3) edge (6)
    (5) edge (4);
;
\end{tikzpicture}
\caption{ The graph $\H$ for $\E_{30}.$}
\end{figure}
\vglue 2mm
\begin{table}[!ht]
\centering
\begin{tabular}{|c|c|c|c|c|}
 \hline
 $s$ & $d_s$ & $[d_s]$ & $N(d_s)$& $\deg(d_s)$\\
 \hline
 000&1&$\{1,7,11,13,17,19,23,29\}$&$[1]\cup [2]\cup [3]\cup [5]\cup[6]\cup[10]\cup [15]$ & $29$\\
 \hline
 100&2&$\{2,4,8,14,16,22,26,28\}$&$[1]\cup [3]\cup [5]\cup [15]$& $15$\\
 \hline
 010&3&$\{3,9,21,27\}$&$[1]\cup [2]\cup [5]\cup [10]$& $20$\\
 \hline
 001&5&$\{5,25\}$&$[1]\cup [2]\cup [3]\cup [6]$&$24$\\
 \hline
 110&6&$\{6,12,18,24\}$&$[1]\cup [5]$&$10$\\
 \hline
 101&10&$\{10,20\}$&$[1]\cup [3]$&$12$\\
 \hline
 011&15&$\{15\}$&$[1]\cup [2]$& $16$\\
 \hline
 111&30&\{0\}&[1]&8\\
 \hline
 \end{tabular}
 \caption{The vertex class distribution and degrees of $\E_{30}.$}
 \end{table}
\end{ex}
To determine the degree of a vertex in \( \E_n \), it is sufficient to calculate the degree of \( d_s \) corresponding to each binary sequence \( s .\) For \( s = (0, 0, \dots, 0) \), the degree is given by \( \deg(d_s) = n - 1 .\) The following theorem provides a general formula for computing the degree of a vertex.
\begin{theorem}
In the graph $\E_n,$ the vertex degree corresponding to every binary sequence $s\neq (0\,0\,\dots\,0)$ is given by $$\deg(d_s)=\frac{n}{d_s}\varphi(d_s).$$
\end{theorem}
\begin{proof}
Computing $\deg(d_s)$ follows from the fact that if $x\in [d_s]$ and $y\in [d_t],$ then $x$ and $y$ adjacent if and only if $\gcd(d_s,d_t)=1.$ Hence
$$\deg(d_s)=\varphi(n)+\sum\limits_{1\neq d_t\in D\left(\frac{n_{0}}{d_{s}}\right)}|[d_t]|,$$ where $D(k)$ denotes the set of divisors of $k.$ On the other hand, a direct approach will give $\deg(d_s)$ in a more compact form.
Let $x\in V$ be adjacent to $d_s\ne 1.$ Then $\gcd(x,d_s)\in U(n)$ and we can write $x=d_su+r,$ where $u\in \N\cup \{0\}$ and $r\in U(d_s).$ Hence 
\begin{equation}\label{eq:deg}
    \deg(d_s)=|\{x\in V|\gcd(x,d_s)\in U(n)\}|=\frac{n}{d_s}\varphi(d_s).
\end{equation}
\end{proof}
\begin{cor}
    The number of edges in a graph $\E_n$ is given by 
    $$|E|=\frac{\varphi(n)}{2}\left(\frac{n}{n_0}\sigma(n_0)-1\right),$$
    where $\sigma(n_0)$ denotes the sum of divisors of $n_0.$
\end{cor}
\begin{proof}
Recall that if $d_s=1,$ then $\deg{(d_s)}=n-1.$ For $d_s\neq 1,$ the $\deg(d_s)=\frac{n}{d_s}\varphi(d_s).$ We already computed the size of an equivalence class, that is, $[d_s]=\frac{n}{n_0}\varphi(\frac{n_0}{d_s}).$ \\ Thus we have
  \begin{align*}
    |E|&=\frac{1}{2}\left(\varphi(n)(n-1)+\sum\limits_{d_s\neq 1}\frac{n}{n_0}\varphi\left(\frac{n_0}{d_s}\right)\frac{n}{d_s}\varphi(d_s)\right).\\
    &=\frac{1}{2}\left(\varphi(n)(n-1)+\sum\limits_{d_s\neq 1}\frac{n}{n_0}\frac{n}{d_s}\varphi(n_0)\right)\\
    &=\frac{1}{2}\left(\varphi(n)(n-1)+\sum\limits_{d_s\neq 1}\frac{n}{d_s}\varphi(n)\right)\\
    &=\frac{\varphi(n)}{2}\left((n-1)+n\sum\limits_{d_s\neq 1}\frac{1}{d_s}\right)\\
    &=\frac{\varphi(n)}{2}\left((n-1)+\frac{n}{n_0}(\sigma(n_0)-n_0)\right)\\
    &=\frac{\varphi(n)}{2}\left(\frac{n}{n_0}\sigma(n_0)-1\right).
\end{align*}
\end{proof}
The following result summarizes all the graph properties of $\E_n.$ We denote the diameter of $\E_n$ by $\text{ diam}(\E_n),$ clique number by ${\omega}(\E_n)$ and the chromatic number by $\chi(\E_n).$ As mentioned earlier, $\omega(n)$ represents the number of distinct prime divisors of $n.$ 
\begin{pro}
    The graph properties of $\E_n$ are as follows.  
    \begin{enumerate}
    \item The graph is connected with diameter $2$ if $n$ is not prime, and diameter $1$ if $n$ is prime.  
    \item The graph is regular if and only if $n$ is prime.  
    \item The graph is bipartite only for $n = 2.$  
    \item  The graph is Hamiltonian for $n > 2.$  
    \item  The graph is Eulerian if and only if $n$ is odd.  
    \item The graph is planar only for $n = 2, 3, 4, 6.$  
    \item Both the chromatic number and clique number are $\varphi(n) + \omega(n).$  
    \item The independence number is $\frac{n}{p_1},$ where $p_1$ is the smallest prime divisor of $n.$  
    \end{enumerate}
\end{pro}
\begin{proof}
\begin{enumerate}
    \item For each $x \in V,$ it is clear that $x$ is adjacent to $1,$ implying that $\E_n$ is connected with the diameter at most $2.$ If $n = p$ (prime), $\text{diam}(\E_n) = 1.$ However, when $n$ is composite, any $1\neq d \mid n$ is not adjacent to $0.$  Therefore, the diameter is exactly 2 in that case.
    \item Note that $\deg(1)\neq \deg(0)$ unless $n$ is a prime. Thus, $\E_n$ is regular if and only $n$ is a prime.
    \item We only need to check the case when $\varphi(n)=2.$ Since $\varphi(n)=2$ if and only if $n\in \{3,4,6\}$ and in all the cases, the graph contains a triangle by $U(n)\cup \{0\}.$ Thus, $\E_n$ is a bipartite graph if and only if $n=2.$ 
    \item It is easy to see that $\{i,i+1\}\in E $ for each $i\in V.$ Thus, $0\backsim 1\backsim 2\backsim \ldots\backsim n-1\backsim 0$ forms a Hamiltonian cycle. 
    \item If $n$ is even, then $\deg(1)$ is odd. For odd $n,$ we have $d_s \neq 2,$ so $\varphi(d_s)$ is always even. Since $\deg(d_s) = \frac{n}{d_s} \varphi(d_s),$ it is always even. Therefore, $\E_n$ is an Eulerian graph if and only if $n$ is odd.
    \item It is clear that $\E_n$ is a planar graph if $n\in \{2,3,4,6\}.$ For the remaining values of $n,$ since $\varphi(n)\geq 4$ hence $\E_n$ contains $K_5$ as a subgraph, for example the subgraph given by the vertices $U(n)\cup \{0\}$ forms $K_5.$ Thus $\E_n$ is a planar graph if and only if $n=2,3,4,6.$  
    \item First note that $\omega(\E_n)\leq \chi(\E_n).$  Clearly, $U(n)\cup \{p_1,p_2,\ldots,p_r\}$ forms a clique of size $\varphi(n)+\omega(n).$ Thus, $\varphi(n)+\omega(n)\leq \omega(\E_n).$ On the other hand, let $p$ be a prime divisor of $n.$ Then all the vertices of an equivalence clases $[d_s],$ where $p\mid d_s$ can be colored with the same vertex color as $p.$ Hence the $\chi(\E_n)\le \varphi(n)+\omega(n).$ Thus, $\omega(\E_n)=\chi(\E_n)=\varphi(n)+\omega(n).$ As a consequence, one can observe that $\E_n$ are perfect graphs.
    \item If $p$ is a prime divisor of $n,$ then $\{px\,|\,0\le x\le \frac{n}{p}-1\}$ forms an independent set. Hence the largest independent set is $\{p_1x|0\le x\le \frac{n}{p_1}-1\},$ where $p_1$ is the least prime. Thus $\al(\E_n),$ the independence number of $\E_n$ is $\frac{n}{p_1}.$
    \end{enumerate}
    \end{proof}

\section{Adjacency Matrix Spectrum}\label{sec:adjacency}
In this section, we focus on establishing bounds for the eigenvalues of the adjacency matrix of the generating graph of finite cyclic groups. As previously noted, for \(\Z_n\), the generating graph is isomorphic to the co-maximal graph. While the spectra of adjacency for co-maximal graphs have been studied in ~\cite{Add}, however a complete characterization remains unresolved in the general case. In this paper, we use matrix analysis techniques to derive bounds for the eigenvalues of the adjacency matrix for arbitrary \(n\) while providing complete solutions for certain special cases.
\subsection{Preliminaries}
Before addressing the core of the section, we first understand the concept of equitable partition in a graph and the results related to the \(\H\)-join of graphs.

For a given graph $\G,$ a partition of the vertex set $V(\G)$ is said to be equitable if for each $i$ and for all $u,v\in V_i,$ $N(u)\cap V_j =N(v)\cap V_j$ for all $j.$ Then we have the following result.
\begin{theorem}[Theorem 4,~\cite{equibook}]
 If $V_1,V_2,\dots V_t$ is an equitable partition of a graph $\G$ and $t_{ij}=|N(v)\cap V_j|$ for all $v\in V_i$ and $T^A$ is the matrix $[t_{ij}],$ called the quotient matrix of the graph, then the characteristic polynomial of $T^A$ divides $\phi_A(x),$ the characteristic polynomial of the adjacency matrix $A.$
\end{theorem}
The concept of an equitable partition of a vertex set in a graph can be generalized to any square matrix. Let \( M \) be an \( n \times n \) square matrix. Suppose the rows and columns of \( M \) can be partitioned into \( k \) subsets \( \{U_1, U_2, \ldots, U_k\} \) such that for each \( i \), the sum of the entries in each block \(M_{U_i, U_j}\)(the submatrix formed by the intersection of rows in \( U_i \) with columns in \( U_j \)) is constant for all \( j .\) In other words, each block \( M_{U_i, U_j} \) has the same row sum. This partition is known as an equitable partition.
The matrix \( \hat{M} \), referred to as the quotient matrix, is defined such that its \( (i, j) \)-th entry represents the constant row sum of the block \( M_{U_i, U_j} .\) A general result states that the characteristic polynomial of \( \hat{M} \) divides the characteristic polynomial of the original matrix \( M .\)
\begin{theorem}[Theorem 7,~\cite{equibook}]\label{thm:ad}
 If $\G_1,\G_2,\dots ,\G_k$ are all $r_i$-regular graphs, then $V(\G_1)\cup V(\G_2)\cup \dots \cup V(\G_k)$ forms an equitable partition of the graph $\H[\G_1,\G_2,\dots,\G_k].$ The quotient matrix $T^A=[t_{ij}]$ is associated with the equitable partition, then the characteristic polynomial of $A$ of the graph $\H[\G_1,\G_2,\dots,\G_k]$ is given by
    $$\phi_A(x)=\phi_{T^A}(x)\prod\limits_{i=1}^{k}\dfrac{\phi_{A_i}(x)}{x-r_i},$$
    where $A_i$ denotes the adjacency matrix of the graph $\G_i.$  
\end{theorem}
With this foundational result, we can now consider our case, where the adjacency matrix $A$ of $\E_n$ is an $n\times n$ matrix whose entries $a_{ij}$ are given by
    \[a_{ij}=
    \begin{dcases}
 1&\quad\gcd(v_{i},v_{j})\in U(n), \\
 0 &\quad\text{otherwise.}\\
    \end{dcases}
    \]

Since the graphs are simple and undirected, the matrix \(A\) for each \(n \in \N\) is symmetric, with all diagonal entries equal to zero. The matrix \(A\) can be expressed in block form by permuting its rows and columns so that elements belonging to the same equivalence class under the relation \(\sim\) are grouped consecutively. In this arrangement, the matrix takes the following form.

 For $n\in \N,$ the adjacency matrix $A=\Big[[d_{ss'}]\Big]$ of the graph $\E_n,$ where $d_{ss'}$ represents the block matrix corresponding to the classes $[d_{s}]$ and $[d_{s'}],$ is given by 
        \[[d_{ss'}]=
        \begin{dcases}
 J&\quad\gcd(d_{s},d_{s'})=1,\\
 J-I&\quad d_{s}=d_{s'}=1, \\
 0&\quad\text{otherwise},\\
        \end{dcases}
        \]
where $J$ and $0$ are the square matrices of order $\Big|[d_{s}]\Big|\times\Big|[d_{s'}]\Big|,$ all of whose entries are 1 and 0 respectively, and $I$ is the identity matrix of size $\Big|[d_{s}]\Big|\times\Big|[d_{s}]\Big |.$
Let us discuss some specific cases.
\begin{ex}
Let $n=p^e,$ where $p$ is a prime and $e\in \N.$ The partition of the vertex set is $\{[1],[p]\}. $ Thus, there are exactly four blocks in the adjacency matrix. Then, $A$ can be written as follows, with rows and columns are ordered as $[1]<[p]$
\begin{table}[!ht]
    \centering
\begin{tabular}{c|cc}
    &$[1]$&$[p]$\\
    \hline
     $[1]$&$J-I$&$J$\\
    $[p]$&$J$&$0$
\end{tabular},
\end{table}\\
where each block size is with respect to the corresponding classes cardinalities. 
\end{ex}
\begin{ex}
Let $n=p_1^{e_1}p_2^{e_2},$ where $p_1,p_2$ are distinct primes. In this case, we have exactly four equivalence classes, that are $\{[1],[p_1],[p_2],[p_1p_2]\}.$ We choose the order for the rows and columns as $[1]<[p_1]<[p_2]<[p_1p_2],$ then we will have 16 blocks. The adjacency matrix is given by
\begin{table}[!ht]
    \centering
    \begin{tabular}{c|cccc}
        &$[1]$&$[p_1]$&$[p_2]$&$[p_1p_2]$\\
        \hline
       $[1]$& $J-I$&J&J&J\\
       $[p_1]$& $J$&0&J&0\\
        $[p_2]$&$J$&$J$&0&0\\
        $[p_1p_2]$& $J$&0&0&0
    \end{tabular},
\end{table}\\
where each block size is with respect to the respective classes cardinalities. 
\end{ex}
Note that for a composite $n\in \N,$ the matrix $A$ is always a singular matrix. This is followed from Lemma~\ref{lem:class}. Moreover, it is easy to see that each block matrix in $A$ has constant row sum. Thus, the relation we initially defined is an equitable partition of $\E_n.$  
Further the graph $\E_n$ can be seen as an $\H$-join of graphs, using Theorem~\ref{thm:ad}, we have the following result.
\begin{theorem}\label{thm:char}
   The characteristic polynomial of the matrix $A$ of  $\E_n,$ under the equitable partition $\{[d_1=1],[d_2],\dots,[d_{2^r}]\}$ of $V,$ where $d_i's$ are the divisors of $n_0$ is given by
       { $$\phi_A(x)=\phi_{T^A}(x)(x+1)^{\varphi(n)-1}{\prod\limits_{a_i}x^{a_i-1}},$$}
   where {$a_i=\Big|[d_{i+1}]\Big|$} for $1\leq i\leq 2^r-1,$ and $\phi_{T^A}(x)$ is the characteristic polynomial of the quotient matrix $T^A.$
      \end{theorem}
      \subsection{Eigenvalue Bounds}
      We will maintain the same notations as previously established. In this subsection, our first objective is to determine the structure of the quotient matrix \( T^A \) for \( \E_n \), as described in Theorem~\ref{thm:char}. To achieve this, it is essential to carefully examine the ordering of the rows and columns. Without loss of generality, let \( p_1 < p_2 < \dots < p_r \) be distinct prime divisors of \( n_0 .\) Now we inductively consider the following ordering of rows and columns of $T^A.$
\begin{itemize}
      \item If \( n_0 = p_1 \), the ordering is \( (1, p_1) .\)  
      \item If \( n_0 = p_1p_2 \), the ordering becomes \( (1, p_1, p_2, p_1p_2) .\)  
      \item If \( n_0 = p_1p_2p_3 \), the ordering extends to \( (1, p_1, p_2, p_1p_2, p_3, p_1p_3, p_2p_3, p_1p_2p_3) .\)  
\end{itemize}    
      In general, for \( n_0 \) with \( r \) distinct prime factors, the first \( 2^{r-1} \) entries of the tuple follow the same ordering as when \( n_0 \) has \( r-1 \) primes. The remaining \( 2^{r-1} \) entries are obtained by multiplying each entry in the first half by \( p_r .\) Without loss of generality, we abbreviated the ordered tuple as $(d_1,d_2,\dots,d_{2^r}).$ Note that for each  $1\leq i\leq 2^r,$ we have the following relation.
      \begin{equation}\label{eq:imp}
        d_id_{2^{r}+1-i}=n_0.
      \end{equation}
      This systematic ordering provides a consistent and structured approach to organizing the rows and columns of the matrix, simplifying further analysis. We will adhere to this sequence for rows and columns unless stated otherwise.

The matrix $T^A,$ with the above ordering of rows and columns, is given by
   {$$T^A=\begin{pmatrix}
      \varphi(n)-1&a_1&a_2&\dots&\dots&a_{2^r-1}\\
      \varphi(n)&0&\theta_{23}&\dots&\dots&\theta_{2(2^r-1)}\\
      \varphi(n)&\theta_{32}&0&\theta_{34}&\dots&\theta_{3(2^r-1)}\\
      \vdots&\vdots&\vdots&\ddots&\vdots\\
      \varphi(n)&\theta_{2^r2}&\theta_{2^r3}&\dots&\dots&0
          \end{pmatrix},$$}
          where {\[ \theta_{ij}=\begin{dcases}
          a_{j-1}&\quad \text{ if } \gcd(d_i,d_j)=1;\\
          0 &\quad \text{ if } \gcd(d_i,d_j)\neq 1.
      \end{dcases}\]}
     and {$a_i=\Big|[d_{i+1}]\Big|$} for $1\leq i\leq 2^r-1.$ Moreover, {$a_1+a_2+\dots +a_{2^r-1}=n-\varphi(n).$}

Let us determine the characteristic polynomial of $A$ for some special cases.
\begin{ex}
    Let $\phi_A(x)$ be the characteristic polynomial of $A$ of the graph $\E_n.$ Then we have 
\begin{enumerate}
	\item When $n=p,$ $\E_n$ is a complete graph. Thus the characteristic polynomial is given by $$\phi_A(x)=(x+1)^{(p-1)}(x-(p-1)).$$ 
	\item When $n=p^e,$ where $e>1,$ then we have
	$$\phi_A(x)=\phi_{T^A}(x)(x+1)^{(\varphi(p^e)-1)} x^{(a_1-1)},$$ where $a_1=p^{e-1}$ and 
	\begin{center}
		$T^A=~\begin{pmatrix}p^e-p^{e-1}-1&p^{e-1}\\p^e-p^{e-1}&0\end{pmatrix}$
	\end{center}
    Thus {$$\phi_{T^A}(x)=x^{2}-(p^e-p^{e-1}-1)x-(p^e-p^{e-1})p^{e-1}.$$}
\item When $n=p_1^{e_1}p_2^{e_2},$ where $p_1,p_2$ are distinct primes, and $e_1+e_2\geq 2.$ 
$$\phi_A(x)=\phi_{T^A}(x)(x+1)^{\phi(n)-1}x^{\sum\limits_{i}(a_i-1)},$$ where $a_1=\frac{n}{n_0}\varphi(p_2), a_2=\frac{n}{n_0}\varphi(p_1)$ and $a_3=\frac{n}{n_0}.$ Thus
$$T^A=\begin{pmatrix} \varphi(n)-1&\frac{n}{n_0}\varphi(p_2)&\frac{n}{n_0}\varphi(p_1)&\frac{n}{n_0}\\
    \varphi(n)&0&\frac{n}{n_0}\varphi(p_1)&0\\
    \varphi(n)&\frac{n}{n_0}\varphi(p_2)&0&0\\
    \varphi(n)&0&0&0
    \end{pmatrix}$$
      $$\phi_{T^A}(x)=x^{4}-(\varphi(n)-1)x^3-\frac{n}{n_0}\varphi(n)(n_0-\varphi(n_0)+1)x^2-\frac{n}{n_0}\varphi(n)\left(\varphi(n)+1\right)x+\left(\frac{n}{n_0}\varphi(n)\right)^2.$$
\end{enumerate}
\end{ex}
Analyzing the matrix directly becomes increasingly challenging as \( n \) increases. However, an alternative approach involves finding a similar matrix that is more manageable. In the paper~\cite{Symm}, a similar result was presented regarding the spectrum of $A$ for the generalized join of graphs, as in Theorem~\ref{thm:ad}. However, in place of $T^A,$ the author associate a symmetric matrix. The symmetric matrix for $A,$ denoted by $Q_n,$ is defined as follows.
$$Q_n=\begin{pmatrix}
    \varphi(n)-1&\rho_{1,2}&\dots &\rho_{1,2^r-1}&\rho_{1,2^r}\\
    \rho_{1,2}&0&\dots&\rho_{2,2^r-1}&\rho_{2,2^r}\\
    \vdots&\vdots&\ddots&\vdots&\vdots\\
    \rho_{1,2^r}&\rho_{2,2^r}&\dots&\rho_{2^r-1,2^r}&0
\end{pmatrix},$$
where $$\rho_{i,j}=\begin{dcases}
    \sqrt{a_{i-1}a_{j-1}} & \text{ if } \gcd(d_i,d_j)=1,\\
    \sqrt{\varphi(n)a_{j-1}} & \text{ if } d_i=1 \text{ and } d_j\neq 1,\\
    \,\,0& \text{otherwise}.
\end{dcases}$$
Now in the following lemma, we will see that both the matrices $Q_n$ and $T^A$ for $A$ are similar.
\begin{pro}
    For a given $n,$ the matrices $T^A$ and $Q_n$ of $A$ are similar.
\end{pro}
\begin{proof}
    Let $$P=\begin{pmatrix}
        \sqrt{\varphi(n)}&0&0&\dots&0\\
        0&\sqrt{a_1}&0&\dots&0\\
        0&0&\sqrt{a_2}&\dots&0\\
        \vdots&\vdots&\vdots&\ddots&\vdots\\
    0&0&0&\dots&\sqrt{a_{2^r-1}}
    \end{pmatrix}.$$
    The matrix $P$ is clearly invertible, and it is easy to check that $PT^AP^{-1}=Q_n.$ Hence they are similar.
\end{proof}
Let $$\tilde Q_n=Q_n-D,$$ where $D=\text{diag}(-1,0,0,\dots,0).$ Then $Q_n=\tilde Q_n+D.$ We can see that a small perturbation in the matrix $Q_n,$ leads to an interesting matrix whose spectrum is easier to find. We have the following result on the spectrum for the matrix $\tilde Q_n.$ Moreover, the following result shows that the symmetric matrices correspond to $n$ and $n_0$ are relatable. We denote associated symmetric matrix of $\E_n$ and $\E_{n_0}$ by $\tilde Q_{n}$ and $\tilde Q_{n_0},$ respectively. 
\begin{pro}
 The relation between the symmetric matrices is given by
 $$\tilde Q_{n}=\frac{n}{n_0}\tilde Q_{n_0}.$$
\end{pro}

\begin{proof}
    For the graph $\E_n,$ $\tilde Q_{n}=Q_n-D.$ That is,
    $$\tilde Q_n=\begin{pmatrix}
        \varphi(n)&\rho_{1,2}&\dots &\rho_{1,2^r-1}&\rho_{1,2^r}\\
        \rho_{1,2}&0&\dots&\rho_{2,2^r-1}&\rho_{2,2^r}\\
        \vdots&\vdots&\ddots&\vdots&\vdots\\
        \rho_{1,2^r}&\rho_{2,2^r}&\dots&\rho_{2^r-1,2^r}&0
    \end{pmatrix}$$
 Using Equation~\ref{eq:csize}, we have $a_l=|d_{l+1}|=\frac{n}{n_0}\varphi\left(\frac{n_0}{d_{l+1}}\right).$ Therefore, $$\rho_{i,j}=\begin{dcases}
        \sqrt{a_{i-1}a_{j-1}}=\frac{n}{n_0}\sqrt{\varphi\left(\frac{n_0}{d_i}\right)\varphi\left(\frac{n_0}{d_j}\right)} & \text{ if } \gcd(d_i,d_j)=1,\\
        \sqrt{\varphi(n)a_{j-1}}=\frac{n}{n_0}\sqrt{\varphi(n_0)\varphi\left(\frac{n_0}{d_j}\right)}& \text{ if } d_i=1 \text{ and } d_j\neq 1,\\
        \quad 0& \text{otherwise}.
    \end{dcases}$$
    Thus, we can factor out the scalar \(\frac{n}{n_0}\) and observe that the remaining terms in each entry correspond to the entries in the case when \(n = n_0.\) This establishes the desired relation
   $$\tilde Q_{n}=\frac{n}{n_0}\tilde Q_{n_0}.$$
\end{proof}
Consequently, the core problem of determining eigenvalue bounds for \(A\) of \(\E_n\) reduces to the case when \(n = n_0.\) Furthermore, we only need to figure out the bounds for the matrix \(Q_{n_0}.\) First we see the structure of $\tilde Q_{n_0}$ with respect to the permutation $(d_1\, d_2 \,d_3\,\dots\, d_{2^r}),$ of rows and columns. Then the matrix $\tilde Q_{n_0}$ is given by
\newpage
        \begin{table}[!ht]
         \centering
        \begin{tabular}{c|cccccccc}
            &$[1]$&$[p_1]$&$[p_2]$&$[p_1p_2]$&$[p_3]$&\dots&$[p_2\dots p_r]$&$[p_1p_2\dots p_r]$\\
            \hline
               $ [1]$& $\varphi(n)$&$\rho_{1,2}$&$\rho_{1,3}$&$\rho_{1,4}$&$\rho_{1,5}$&\dots &$\rho_{1,2^r-1}$&$\rho_{1,2^r}$\\

                $[p_1]$&$\rho_{1,2}$&0&$\rho_{2,3}$&$0$&$\rho_{2,5}$&\dots&$\rho_{2,2^r-1}$&0\\

                $[p_2]$&$\rho_{1,3}$&$\rho_{2,3}$&0&$0$&$\rho_{3,5}$&\dots&0&0\\

                $[p_1p_2]$&$\rho_{1,4}$&0&$0$&$0$&$\rho_{4,5}$&\dots&0&$0$\\
                $[p_3]$&$\rho_{1,5}$&$\rho_{2,5}$&$\rho_{3,5}$&$\rho_{4,5}$&0&\dots&0&0\\
                \vdots&\vdots&\vdots&\vdots&\vdots&\vdots&$\ddots$&\vdots&\vdots\\
$[p_2\dots p_r]$&$\rho_{1,2^r-1}$&$\rho_{2,2^r-1}$&0&0&$0$&\dots&0&0\\
              $[p_1p_2\dots p_r]$& $\rho_{1,2^r}$&$0$&$0$&$0$&0&\dots&0 &0 
        \end{tabular}
        \end{table}
In the following result, we determine the spectrum of $\tilde Q_{n_0}.$
\begin{theorem}
    The matrix $\tilde Q_{n_0}=\sqrt {\varphi(n_0)}\Bigg(U_1\bigotimes U_2\bigotimes \dots \bigotimes U_r\Bigg),$ where $
    \bigotimes$ denotes the tensor product of matrices, and $U_i=\begin{pmatrix}
        \sqrt{\varphi(p_{r-{(i-1)}})}&1\\
        1&0
    \end{pmatrix}$ 
with $p_j's$ are primes and $ 1\leq j\leq r.$ 
\end{theorem}
\begin{proof}
    It is important to note that tensor product of symmetric matrices is symmetric.
    Under the given ordered tuple $(d_1,\, d_2 ,\,\dots,\, d_{2^r})$ of divisors of $n_0,$ we will see that $\tilde Q_{n_0}$ is in the following form.
    $$\tilde Q_{n_0}=\sqrt{\varphi(n_0)}\begin{pmatrix} B_{11}&B_{12}\\
        B_{21}&B_{22}
    \end{pmatrix}=\begin{pmatrix}
        *&*\\
        *&0
    \end{pmatrix},$$
where $*$ represents a non-zero block matrix and 0 represents a zero matrix. First note that for every $i,j$ such that $\rho_{i,j}=\sqrt {a_{i-1} a_{j-1}}$ when $\gcd(d_i,d_j)=1.$ Then clearly, 
\begin{equation}\label{eq:imp2}
    a_{i-1}a_{j-1}=\varphi\left(\frac{n_0}{d_i}\right)\varphi\left(\frac{n_0}{d_j}\right)=\varphi(n_0)\varphi\left(\frac{n_0}{d_id_j}\right).
\end{equation}
Hence we can always take $\sqrt{\varphi(n_0)}$ as a common value. Moreover, the entries in the skew diagonal of $\tilde Q_{n_0},$ {\it i.e.,} 
\begin{equation}\label{eq:imp3}
    \rho_{i,2^{r}+1-i}=\sqrt{\varphi(n_0)}
\end{equation} 
for every $i,$ $1\leq i\leq 2^r.$

 In the tuple $(d_1,\,d_2,\,\dots,\, d_{2^r}),$ we first subdivide it into two subtuples $(d_1,\,d_2,\,\dots, d_{2^{r-1}}),$ and $(d_{2^{r-1}+1},\,\dots\, ,d_{2^r}).$ Then we see that the matrix $\tilde Q_{n_0}$ can be partitioned into four blocks, with $B_{22}$ clearly a zero matrix because $p_r$ is common in the set $\{d_{2^{r-1}+1},\,\dots\, ,d_{2^r}\}.$ 
 Let $B=\begin{pmatrix}
    B_{11}&B_{12}\\
    B_{21}&0
\end{pmatrix}.$
We will apply induction on $r,$ to get an expression for $B$ in the tensor products of $2\times 2$ matrices of type $\begin{pmatrix}
    *&1\\
    1&0
\end{pmatrix}.$ 
We will follow the following inductive steps.
\begin{enumerate}
    \item Let $r=1,$ then $n_0=p_1$ and we have $$B=\begin{pmatrix}
        \sqrt{\varphi(p_1)}&1\\
        1&0
    \end{pmatrix}$$
    Hence $B=U_1.$
    \item For $r=2,$ $n_0=p_1p_2,$ then we have $$B=\begin{pmatrix}
        \sqrt{\varphi(p_1)\varphi(p_2)}& \sqrt{\varphi(p_2)}& \sqrt{\varphi(p_1)}&1\\
        \sqrt{\varphi(p_2)}&0&1&0\\
        \sqrt{\varphi(p_1)}&1&0&0\\
        1&0&0&0
    \end{pmatrix}=
    \begin{pmatrix}
        \sqrt{\varphi(p_2)}&1\\
        1&0
    \end{pmatrix}\otimes \begin{pmatrix}
        \sqrt{\varphi(p_1)}&1\\
        1&0
    \end{pmatrix}=U_1\otimes U_2.$$
\end{enumerate}
Thus by induction hypothesis,
it holds for $r-1.$ For $n_0=p_1p_2\dots p_r,$ the matrix $U_1\otimes U_2\otimes \dots \otimes U_{r}$ 
can be seen as  
$$\begin{pmatrix}
            \sqrt{\varphi(n_0)}&\sqrt{\varphi\left(\frac{n_0}{d_2}\right)}&\sqrt{\varphi\left(\frac{n_0}{d_3}\right)}&\sqrt{\varphi\left(\frac{n_0}{d_4}\right)}&\sqrt{\varphi\left(\frac{n_0}{d_5}\right)}&\dots &\sqrt{\varphi\left(\frac{n_0}{d_{2^{r}-1}}\right)}&\sqrt{\varphi\left(\frac{n_0}{d_{2^r}}\right)}\\

            \sqrt{\varphi\left(\frac{n_0}{d_{2}}\right)}&0&\sqrt{\varphi\left(\frac{n_0}{d_2d_{3}}\right)}&0&\sqrt{\varphi\left(\frac{n_0}{d_2d_{5}}\right)}&\dots&\sqrt{\varphi\left(\frac{n_0}{d_2d_{2^r-1}}\right)}&0\\

            \sqrt{\varphi\left(\frac{n_0}{d_{3}}\right)}&\sqrt{\varphi\left(\frac{n_0}{d_{3}d_2}\right)}&0&0&\sqrt{\varphi\left(\frac{n_0}{d_{3}d_5}\right)}&\dots&0&0\\

            \sqrt{\varphi\left(\frac{n_0}{d_{4}}\right)}&0&0&0&\sqrt{\varphi\left(\frac{n_0}{d_4d_{5}}\right)}&\dots&0&0\\
            \sqrt{\varphi\left(\frac{n_0}{d_{5}}\right)}&\sqrt{\varphi\left(\frac{n_0}{d_5d_{2}}\right)}&\sqrt{\varphi\left(\frac{n_0}{d_5d_{3}}\right)}&\sqrt{\varphi\left(\frac{n_0}{d_5d_{4}}\right)}&0&\dots&0&0\\
      \vdots&\vdots&\vdots&\vdots&\vdots&\ddots&\vdots&\vdots\\
      \sqrt{\varphi\left(\frac{n_0}{d_{2^r-1}}\right)}&\sqrt{\varphi\left(\frac{n_0}{d_{2^r-1}d_2}\right)}&0&0&0&\dots&0&0\\
      \sqrt{\varphi\left(\frac{n_0}{d_{2^r}}\right)}&0&0&0&0&\dots&0 &0 
\end{pmatrix}$$
From Equation~\ref{eq:imp}, each skew diagonal entry is equal to 1. Now we can easily see that
    $$U_1\otimes U_2\otimes \dots \otimes U_{r-1}\otimes U_r=B,$$ where $$U_i=\begin{pmatrix}
        \sqrt{\varphi(p_{r-{(i-1)}})}&1\\
        1&0
    \end{pmatrix}$$ for $1\leq i\leq r.$  
\end{proof}
\begin{cor}
    The eigenvalues of \(\tilde Q_{n_0}\) are given by \(\sqrt{\varphi(n)} \prod\limits_{i=1}^{r} u_{i,l}\), where \(u_{i,l}\) represents an eigenvalue of matrix \(U_i\), with \(l = 1, 2\), as follows.
     $$u_{i,1}=\dfrac{\sqrt{\varphi\left(p_{r+1-i}\right)}+ \sqrt{4+\varphi\left(p_{r+1-i}\right)}}{2} \text{ and } u_{i,2}=\dfrac{\sqrt{\varphi\left(p_{r+1-i}\right)}-\sqrt{4+\varphi\left(p_{r+1-i}\right)}}{2}.$$
\end{cor}
\begin{proof}
    In the theorem above, we established that \(\tilde{Q}_{n_0}=\sqrt{\varphi(n_0)}\,B\) where $B$ is the tensor product of $U_i's.$ Utilizing the property that the eigenvalues of a tensor product of matrices are the products of the eigenvalues of the individual matrices, it follows that the eigenvalues of \(B\) are the products of the eigenvalues of the \(U_i.\) Since $U_i=\begin{pmatrix}
        \sqrt{\varphi(p_{r-{(i-1)}})}&1\\
        1&0
    \end{pmatrix}$ 
    for $1\leq i\leq (r-1).$ Thus for each $i,$ the eigenvalues of $U_i$ are given by
    $$u_{i,1}=\dfrac{\sqrt{\varphi\left(p_{r+1-i}\right)}+ \sqrt{4+\varphi\left(p_{r+1-i}\right)}}{2} \text{ and } u_{i,2}=\dfrac{\sqrt{\varphi\left(p_{r+1-i}\right)}-\sqrt{4+\varphi\left(p_{r+1-i}\right)}}{2}.$$

\end{proof}
Note that $Q_n$ is a perturbation of the matrix $\tilde Q_n,$ that is, $Q_n=\tilde Q_n-D.$  Since $\tilde Q_n$ and $D$ are non-commuting, thus it is not easy to find the exact eigenvalues of $Q_n,$ however we can use Weyl's inequalities to get the eigenvalue bounds. 
\begin{theorem}[\cite{bhatia}]
   Let $\lambda_1(M)\geq \lambda_2(M)\geq \lambda_3(M)\geq \dots\geq \lambda_m(M)$ be the eigenvalues in the decreasing order of a matrix $M.$ Let $A$ and $B$ be $m\times m$ Hermitian matrices.  Then 
   \begin{align*}
    \lambda_j(A+B)\leq \lambda_i(A)+\lambda_{j-i+1}(B) &\text{ for } i\leq j,\\
   \lambda_j(A+B)\geq \lambda_i(A)+\lambda_{j-i+n}(B)&\text{ for } i\geq j.
   \end{align*}
\end{theorem}
\begin{cor}[\cite{bhatia}]\label{cor:bound}
    For each $j=1,2,\dots,m,$ we have the following inequality.
    $$\lambda_j(A)+\lambda_m(B)\leq \lambda_j(A+B)\leq \lambda_j(A)+\lambda_1(B).$$
\end{cor}
\begin{theorem}
   The eigenvalue bounds for $Q_n$ is given by 
   $$\lambda_j(\tilde Q_n)+\lambda_n(D)\leq \lambda_j( Q_n)\leq \lambda_j(\tilde Q_n)+\lambda_1(D),$$ where $\tilde Q_n=Q_n-D,$ and $D=\text{diag}(-1,0,0,\dots,0).$
\end{theorem}
\begin{proof}
We know that the eigenvalues of $D,$ that is,
$0\geq0\geq\dots \geq 0\geq -1.$ Then by the above result, we have the following inequality
$$\lambda_j(\tilde Q_n)-1\leq \lambda_j(Q_n)\leq \lambda_j(\tilde Q_n).$$
\end{proof}
To understand the above discussion, Let us illustrate it in the following example.
\begin{ex}
    Let $n=15,$ where $p_1=3, p_2=5.$
   Then we have $$\tilde Q_n=\sqrt{8}\begin{pmatrix}
    2&1\\
    1&0
\end{pmatrix}\otimes \begin{pmatrix}
    \sqrt{2}&1\\
    1&0
\end{pmatrix}=\sqrt{8}\left(U_1\otimes U_2\right)$$
and the eigenvalues $u_{1,l}=\dfrac{2\pm \sqrt{8}}{2}$; $u_{2,l}=\dfrac{\sqrt{2}\pm\sqrt{6}}{2};$ The eigenvalues of $\tilde Q_n$ are $\sqrt{8}\prod\limits_{i=1}^{2}u_{i,l},$ where $l=1,2.$  So, we have the eigenvalues table (approximate) in the decreasing sequence, as follows.
\begin{table}[!ht]
    \centering
    \begin{tabular}{|c|c|c|c|}
        \hline
        &$\lambda_i(\tilde Q_n)-1$&$\lambda_i(Q_n)$&$\lambda_i(\tilde Q_n)$\\
        \hline
        1&12.1915&12.5326&$13.1915$\\
        \hline
        2&-0.393549&0.576439&$0.606451$\\
        \hline
        3&-3.2633&-2.36821&$-2.2633$\\
        \hline
       4&-4.53465&-3.74081& $-3.53465$\\
        \hline
\end{tabular}
\caption{Eigenvalues and its bounds for the matrix $Q_{15}.$}
\end{table}
Since $Q_n=\tilde Q_n+D,$ where $D=\text{diag}(-1,0,0,0).$ The above table verified the previous result.
$$\lambda_j(\tilde Q_n)-1\leq \lambda_j(Q_n)\leq \lambda_j(\tilde Q_n).$$
\end{ex}
\section{Laplacian Matrix Spectrum}\label{sec:laplacian}
For a graph \( \mathcal{G} \) with a vertex set of size \( m \), the Laplacian matrix is defined as \( L(\mathcal{G}) = D(\mathcal{G}) - A(\mathcal{G}) \), where \( D(\mathcal{G}) = \text{diag}(\alpha_1, \alpha_2, \dots, \alpha_m) \) is the diagonal matrix of vertex degrees, and \( A(\mathcal{G}) \) is the adjacency matrix. In the studies~\cite{Lap1,Lap2}, the authors investigated the Laplacian spectrum of the co-maximal graph of \( \mathbb{Z}_n .\) While they analyzed the integrability of the Laplacian matrix, the exact spectrum was not explicitly provided. In this work, we establish bounds for the Laplacian spectrum and provide additional eigenvalues that were not explicitly listed in their findings.

One can see that Laplacian matrix of the graph $\E_n$ is an $n\times n$ matrix $L=\Big[[d_{ij}]\Big],$ where $d_{ij}$ represents the block matrix corresponding to the classes $[d_{i}]$ and $[d_{j}],$ and $i,j\in \{1,2,3,\dots, 2^r\}.$ Then
		\[[d_{ij}]=
		\begin{dcases}
		-J&\quad\gcd(d_{i},d_{j})=1;\\
		nI-J&\quad d_{i}=d_{j}=1; \\
        \deg(d_i)I&\quad d_i=d_{j}\neq 1;\\
	     0&\quad\text{otherwise,}
		\end{dcases}
		\]
where $J$ and $0$ are square matrices of order $\Big|[d_{i}]\Big|\times\Big|[d_{j}]\Big|,$ all of whose entries are 1 and 0 respectively, $I$ is the identity matrix of order $\Big|[d_{i}]\Big|\times \Big|[d_{i}]\Big|.$

Each of the blocks has the same row sums, thus the partition of vertices given as an equitable partition. Then the quotient matrix of $L$ is given by 
\begin{equation}\label{eq:M}
M=
\begin{pmatrix}
n-\varphi(n)&-a_1&-a_2&\dots&-a_{2^r-2}&-a_{2^r-1}\\
-\varphi(n)&\deg(d_2)&-\theta_{23}&\dots&-\theta_{2(2^r-1)}&0\\
-\varphi(n)&-\theta_{32}&\deg(d_3)&\dots&-\theta_{3(2^r-1)}&0\\
\vdots&\vdots&\vdots&\ddots&\vdots&\vdots\\
-\varphi(n)&-\theta_{(2^r-1)2}&-\theta_{(2^r-1)3}&\dots&\deg(d_{2^r-1})&0\\
-\varphi(n)&0&0&\dots&0&\varphi(n)
\end{pmatrix},
\end{equation}
where {\[ \theta_{ij}=\begin{dcases}
    a_{j-1}=\dfrac{n}{n_0}\varphi\left(\frac{n_0}{d_{j}}\right)&\quad \text{ if } \gcd(d_i,d_j)=1;\\
    0 &\quad \text{ if } \gcd(d_i,d_j)\neq 1.
\end{dcases}\]}
Moreover, we have  
\begin{equation}\label{eq:lap2}
    a_1+a_2+\dots +a_{2^r-2}+a_{2^r-1}=n-\varphi(n).
\end{equation}
Recall that the characteristic polynomial of $M$ divides $\phi_L(x)$. However, determining all the eigenvalues of $M$ is not straightforward. We will explore alternative methods to identify the eigenvalues of $M$. 

To begin, we present a result that provides the Laplacian eigenvalues of a generalized join of a family of graphs. This result involves a symmetric matrix associated with the Laplacian matrix, which we later demonstrate to be similar to $M$ in our case.
\begin{theorem}[Theorem 8,~\cite{Symm}]\label{Theo:LSQ}
    Let $\mathcal{F}$ be a family of $k$-graphs $\G_j$ of order $n_j$ with $j\in\{1,2,\dots, k\},$ with Laplacian spectrum $\sigma_L(\G_j).$ If $\H$ is a graph such that $V(\H)=\{1,2,\dots,k\},$  then the Laplacian spectrum of $\H[\G_1,\dots,\G_k]$ is given by
    $$\sigma_L\left(\H[\G_1,\dots,\G_k]\right)=\left(\bigcup\limits_{i=1}^{k}N_j+\left(\sigma_L(G_j)\setminus\{0\}\right)\right)\bigcup \sigma(L_Q),$$ where 
 $L_Q$ is the symmetric matrix
    $$L_Q=\begin{pmatrix}
       N_1&-\rho_{1,2}&\dots &-\rho_{1,k-1}&-\rho_{1,k}\\
        -\rho_{1,2}&N_2&\dots&-\rho_{2,k-1}&-\rho_{2,k}\\
        \vdots&\vdots&\ddots&\vdots&\vdots\\
        -\rho_{1,k}&-\rho_{2,k}&\dots&-\rho_{k-1,k}&N_k
    \end{pmatrix},$$
and  $\rho_{i,j}=\begin{dcases}
    \sqrt{n_in_j} & \text{ if } ij\in E(\H)\\
    0& \text{otherwise}
\end{dcases};$
$N_j=\begin{dcases}
    \sum\limits_{i\in N_H(j)}n_i & \text{ if } N_{\H}(j)\neq 0,\\
 0 & \text{otherwise}.
\end{dcases}$
\end{theorem}
Now we apply Theorem~\ref{Theo:LSQ}
to get the spectrum of the matrix $L$ of $\E_n.$
\begin{theorem}\label{thm:lapS}
    The Laplacian spectrum of $\E_n=\H[\G_1,\G_2,\dots, \G_{2^r}],$ where $r$ is the number of distinct prime factors of $n,$ and $\G_j$ are the induced subgraphs of order $n_j$ with $j\in\{1,2,\dots, 2^r\}$ is given by
    $$\sigma_{L}(\E_n)=\Bigg\{n^{(\varphi(n)-1)},\,\deg(d_2)^{(a_1-1)},\,\deg(d_3)^{(a_2-1)},\dots,\deg(d_{2^r-1})^{(a_{2^r-2})},\, \varphi(n)^{(a_{2^r-1})}\Bigg\} \bigcup \sigma(L_Q),$$ 
where the symmetric matrix $L_Q$ is as follows:
$$L_Q=\begin{pmatrix}
       n-\varphi(n)&-\rho_{1,2}&\dots &-\rho_{1,2^r-1}&-\rho_{1,2^r}\\
        -\rho_{1,2}&\deg(d_2)&\dots&-\rho_{2,2^r-1}&-\rho_{2,2^r}\\
        \vdots&\vdots&\ddots&\vdots&\vdots\\
        -\rho_{1,2^r}&-\rho_{2,2^r}&\dots&-\rho_{2^r-1,2^r}&\varphi(n)
    \end{pmatrix},$$
and $\rho_{i,j}=\begin{dcases}
    \sqrt{a_{i-1}a_{j-1}} & \text{ if } \gcd(d_i,d_j)=1,\\
    \sqrt{\varphi(n)a_{j-1}} & \text{ if } d_i=1 \text{ and } d_j\neq 1,\\
    0& \text{ otherwise.}
\end{dcases}$
\end{theorem}
\begin{proof}
   Recall that $\G_1$ is $\varphi(n)-1$ regular graph, and other $\G_i's$ are 0-regular graphs. In addition, $n_1=\varphi(n)$ and $n_i=a_{i-1},$ where $2\leq i\leq 2^r.$ Thus the Laplacian matrix of $\G_1,$ denoted by $L(\G_1)=\varphi(n)I-(J-I).$ Therefore, $\sigma_{L(\G_1)}=\Big\{\varphi(n)^{(\varphi(n)-1)},0^{(1)}\Big\},$ however, $\sigma_{L(\G_i)}=\Big\{0^{(a_{i-1})}\Big\},$ where the exponent represent the multiplicity of the eigenvalues. Since $N_i$ represents the size of the neighbourhood set of the graph $\G_i$ in $\E_n.$ Clearly, $N_1=n-\varphi(n),$ and other $N_i's$ are given by $N_i=\deg(d_i).$ Moreover, $V(\G_1)\cup V(\G_2)\cup\dots V(\G_{2^r})$ is an equitable partition of $V(\E_n),$ and $L_Q$ is the associated symmetric matrix. Applying Theorem~\ref{Theo:LSQ}, we get the desired result.
\end{proof}
\begin{theorem}
    For a given $n,$ the matrices $L_Q$ and $M$ of $L$ are similar. Then $0,\varphi(n),n$ are some eigenvalues of $L_Q.$ 
\end{theorem}
\begin{proof}
    Let $$P=\begin{pmatrix}
        \sqrt{\varphi(n)}&0&0&\dots&0\\
        0&\sqrt{a_1}&0&\dots&0\\
        0&0&\sqrt{a_2}&\dots&0\\
        \vdots&\vdots&\vdots&\ddots&\vdots\\
    0&0&0&\dots&\sqrt{a_{2^r-1}}
    \end{pmatrix}.$$
    The matrix $P$ is clearly invertible, and thus $PMP^{-1}=L_Q.$ Hence they are similar. 

    Recall that the row sum of $L$ is zero, hence $M$ has the row sum zero. Then the principal submatrix $M[\{2,3,\dots, 2^r-1\};\{2,3,\dots,2^r-1\}],$ which is of size $(2^r-2)\times (2^r-2)$ has row sum $\varphi(n)$(see~\ref{eq:M}), and since $a_{2^r-1}=\frac{n}{n_0},$ then by Equation~\ref{eq:lap2}, we have 
    $$-(a_1+a_2+\dots+a_{2^r-2})=-\left(n-\varphi(n)-\frac{n}{n_0}\right).$$
    Thus we can get the quotient matrix of $M$ as well, which is given by 
    $$\hat M=\begin{pmatrix}n-\varphi(n)&-\left(n-\varphi(n)-\frac{n}{n_0}\right)&-\frac{n}{n_0}\\
    -\varphi(n)&\varphi(n)&0\\-\varphi(n)&0&\varphi(n)\end{pmatrix}.$$
    It is easy to find the eigenvalues of the matrix $\hat M,$ that are $0,\varphi(n) \text{ and } n.$ Because of the  similarity of matrices, these are also the eigenvalues of $L_Q.$
\end{proof}
Since it is equivalent to determining the spectrum of the matrices \( L_Q \) and \( M .\) We still haven't get all the eigenvalues, as we can see the remaining eigenvalues proves to be challenging. Therefore, we aim to establish bounds for the eigenvalues of \( L_Q .\)
Note that we can write $$L_Q=-\tilde Q_n+D_{L},$$ where $D_{L}=\text{diag}(n,\deg(d_2),\deg(d_3),\dots,\deg(d_{2^r-1}),\varphi(n))$ and $\tilde Q_n$ is defined in previous section. We apply the Weyl's inequality to find the eigenvalues bounds for the matrix $L_Q.$ 

\begin{theorem}
    The eigenvalues bound for the matrix $L_Q$ of $L$ is given by
    $$\lambda_j(-\tilde Q_n)+\varphi(n)\leq \lambda_j(L_Q)\leq \lambda_j(-\tilde Q_n)+n,$$ where $1\leq j\leq 2^r.$
\end{theorem}
 \begin{proof} 
   We have $L_Q=-\tilde Q_n+D_{L}.$ We know the eigenvalues of $D_{L},$ in which the least value is $\varphi(n)$ and the greatest is $n.$ Then we apply Corollary~\ref{cor:bound}, we get
   $$\lambda_j(-\tilde Q_n)+\lambda_n(D_{L})\leq \lambda_j (L_Q)\leq \lambda_j(-\tilde Q_n)+\lambda_1(D_{L}).$$ Thus, $$\lambda_j(-\tilde Q_n)+\varphi(n)\leq \lambda_j(L_Q)\leq \lambda_j(-\tilde Q_n)+n,$$ where $1\leq j\leq 2^r.$
 \end{proof}
 Let us illustrate it in the following example. 
 \begin{ex}
    Let $n=15,$ where $p_1=3, p_2=5.$
   Then we have $$\tilde Q_n=\sqrt{8}\begin{pmatrix}
    2&1\\
    1&0
\end{pmatrix}\otimes \begin{pmatrix}
    \sqrt{2}&1\\
    1&0
\end{pmatrix}=\sqrt{8}\left(U_1\otimes U_2\right)$$
and the eigenvalues $u_{1,l}=\dfrac{2\pm \sqrt{8}}{2}$; $u_{2,l}=\dfrac{\sqrt{2}\pm\sqrt{6}}{2};$ The eigenvalues of $\tilde Q_n$ are $\sqrt{8}\prod\limits_{i=1}^{2}u_{i,l},$ where $l=1,2.$  So we have the eigenvalues in the decreasing sequence as follows.
\newpage
\begin{table}[!ht]
    \centering
    \begin{tabular}{|c|c|c|c|}
        \hline
        &$\lambda_i(-\tilde Q_n)+8$&$\lambda_i(L_Q)$&$\lambda_i(-\tilde Q_n)+15$\\
        \hline
        1&11.5346&15&$18.5346$\\
        \hline
        2&10.2633&14&$17.2633$\\
        \hline
        3&7.39355&8&$14.39355$\\
        \hline
       4&-5.1915&0& $1.8085$\\
        \hline
\end{tabular}
\caption{Eigenvalues and its bounds for the matrix $L_Q.$}
\end{table}
The following observations can be made from the above table.
$$\lambda_j(-\tilde Q_n)+8\leq \lambda_j(L_Q)\leq \lambda_j(-\tilde Q_n)+15.$$
\end{ex}
\begin{theorem}
  The relation between the matrices $L_{Q(n)}$ and $L_{Q{(n_0)}}$ of $\E_n$ and $\E_{n_0},$  respectively is given by
        $$L_{Q(n)}=\dfrac{n}{n_0} L_{Q{(n_0)}}.$$
\end{theorem}
\begin{proof}
    Note that $$D_{L}=\text{diag}(n,\deg(d_2),\deg(d_3),\dots,\deg(d_{2^r-1}),\varphi(n)).$$ From Equation~\ref{eq:deg}, for $d_i\neq 1,$ we have
    $$\deg(d_i)=\frac{n}{d_i}\varphi(d_i).$$ So,
    we can take $\frac{n}{n_0}$ as a common factor from all the diagonal entries of $D_L.$ Also we have $\tilde Q_{n}=\frac{n}{n_0}\tilde Q_{n_0}.$ Thus,  the $L_{Q(n)}$ can be simplified as $$L_{Q(n)}=-\tilde Q_{n}+D_{L}=\frac{n}{n_0}(-\tilde Q_{n_0}+D_{L(n_0)})=L_{Q(n_0)}.$$
\end{proof}
\begin{cor}
    The Laplacian eigenvalues of the graph \(\E_{n}\) can be obtained by simply multiplying the Laplacian eigenvalues of \(\E_{n_0}\) by \(\frac{n}{n_0}.\)
\end{cor}
\begin{proof}
This is the direct consequence of the above theorem and Theorem~\ref{thm:lapS}.
\end{proof}
Therefore, the problem essentially reduces to solving for \(n_0.\)
 \section{Conclusion}
 In the paper, we introduce a refined partition of the set \( \{0, 1, 2, \dots, n-1\} \), which not only facilitates the study of the structure of \( \mathbb{Z}_n \) but also its generating graph. During our investigation of \( \E_n \), we observed that the co-maximal graphs of a ring exhibit overlap in the case of \( \mathbb{Z}_n.\) While some studies have explored the spectra of adjacency and Laplacian matrices of co-maximal graphs, the problem remains only partially resolved. Utilizing matrix analysis tools, we provide tighter bounds for the eigenvalues, advancing the understanding of this problem.
 \section* {Statements and Declarations}
\subsection* {Competing interests} The authors declare that they have no competing interests.
\subsection* {Author contributions} All the authors have the same amount of contribution.
\subsection* {Funding} The first named author is supported by Shiv Nadar Institution of Eminence Ph.D. Fellowship.
\subsection*{Availability of data and materials} Data sharing does not apply to this article as no data sets
were generated or analysed during the current study.

 \section*{Acknowlegdement}
 The authors sincerely express their gratitude to the anonymous reviewers for their constructive feedback and insightful suggestions, which have significantly enhanced the quality of this work. Additionally, authors extend special thanks to Krishna Kumar Gupta, a fellow PhD student for his helpful suggestions.
   
 \addcontentsline{toc}{section}{Bibliography}
 \bibliographystyle{unsrt}
 \bibliography{refer}
\end{document}